\documentclass{amsart}

\usepackage{amsmath,amsthm, amssymb}
\usepackage[all]{xy}
\usepackage{enumerate}
\usepackage{wasysym}
\usepackage{graphicx,color}

\newcommand{\R}{\mathbb{R}}
\newcommand{\Z}{\mathbb{Z}}
\newcommand{\T}{\mathbb{T}}
\newcommand{\tor}{\mathbb{T}}

\newcommand{\qq}{\mathbf{1}}
\newcommand{\qi}{\mathbf{i}}
\newcommand{\qj}{\mathbf{j}}
\newcommand{\qk}{\mathbf{k}}
\newcommand{\pp}{\mathrm{pr}}

\DeclareMathOperator{\aut}{Aut}
\newtheorem{theorem}{Theorem}[section]

\newtheorem{corollary}[theorem]{Corollary}

\theoremstyle{definition}
\newtheorem{definition}[theorem]{Definition}
\theoremstyle{remark}

\numberwithin{equation}{section}
\title{Examples of $3$-quasi-Sasakian manifolds}

\author[B. Cappelletti Montano]{Beniamino Cappelletti-Montano}
 \address{Dipartimento di Matematica e Informatica, Universt\`a degli Studi di
 Cagliari, Via Ospedale 72, 09124 Cagliari}
 \email{b.cappellettimontano@gmail.com}

\author[A. De Nicola]{Antonio De Nicola}
 \address{CMUC, Department of Mathematics, University of Coimbra, 3001-501 Coimbra, Portugal}
 \email{antondenicola@gmail.com}

\author[I. Yudin]{Ivan Yudin}
 \address{CMUC, Department of Mathematics, University of Coimbra, 3001-501 Coimbra, Portugal}
 \email{yudin@mat.uc.pt}


\thanks{Research partially supported by CMUC, funded by the European program COMPETE/FEDER,
by FCT (Portugal) grants PEst-C/MAT/UI0324/2011 (A.D.N. and I.Y.),  by Prin 2010/11 -- Variet\`{a} reali e complesse:
geometria, topologia e analisi armonica -- Italy (B.C.M.) and by the exploratory research project FCT IF/00016/2013 in the frame of Programa Investigador  (I.Y.).}

\begin{document}

%\begin{center}
\dedicatory{\emph{Dedicated to Prof. Anna Maria Pastore on occasion of her 70th birthday}}
%\end{center}

\begin{abstract}
We provide a general method to construct examples of quasi-Sasakian $3$-structures on a $(4n+3)$-dimensional manifold. Moreover, among this class, we give the first explicit example of a compact $3$-quasi-Sasakian manifold which is not the global product of a $3$-Sasakian manifold and a hyper-K\"{a}hler manifold.
\end{abstract}

\maketitle

\section{Introduction}

The class of  quasi-Sasakian manifolds was introduced by Blair in \cite{blair0}, and then studied by several authors (e.g. \cite{tanno}, \cite{olszak2}, \cite{kanemaki1}) in order to unify the most important classes of almost contact metric manifolds, namely the Sasakian and coK\"{a}hler ones, which are   quasi-Sasakian manifolds of maximal and minimal rank, respectively. Moreover any quasi-Sasakian manifold is canonically endowed with a transversely K\"{a}hler foliation, so that they can be thought as an odd-dimensional analogue of K\"{a}hler manifolds.

When on a smooth manifold $M$ there are defined three distinct quasi-Sasakian
structures, with the same compatible metric, which are related to each other by
certain relations similar to the quaternionic identities, one says that
$M$ is a $3$-quasi-Sasakian manifold  (see Section \ref{preliminari} for
the precise definition). The class  of $3$-quasi-Sasakian manifolds was extensively studied a few years ago
in \cite{CapDenDil1} and \cite{CapDenDil2}, where several properties on
$3$-quasi-Sasakian manifolds, which do not hold for a general quasi-Sasakian
structure, were proved. In particular, it was proved that the aforementioned quaternionic-like
structure forces any $3$-quasi-Sasakian manifold of non-maximal rank $4l+3$ to
be the local Riemannian product of a $3$-$c$-Sasakian manifold and a
hyper-K\"{a}hler manifold. Therefore a natural question arises: are there
examples of $3$-quasi-Sasakian manifolds which are not the global product of a
$3$-$c$-Sasakian manifold and a hyper-K\"{a}hler manifolds? In this article we
give an affirmative answer to this problem. We present a general procedure to
produce a large class of examples, and we prove that the $11$-dimensional
3-quasi-Sasakian manifold in this class  is not a global product of $3$-Sasakian
and hyper-K\"ahler manifolds.

\bigskip

All manifolds considered in this paper will be assumed to be smooth and connected. For wedge product, exterior derivative and interior product we use the conventions as in Goldberg's book \cite{Go}.

\section{Preliminaries}\label{preliminari}
We start with a few preliminaries on almost contact metric manifolds, referring the reader  to the monographs \cite{blair1}, \cite{galicki} and to the survey \cite{CapDenYud0} for further details.

An \emph{almost contact metric structure} on a $(2n+1)$-dimensional
manifold $M$ is the data of a  $(1,1)$-tensor $\phi$, a vector field $\xi$, called  \emph{Reeb vector field},
a $1$-form $\eta$ and a Riemannian metric $g$ satisfying
\begin{gather}\label{defacm}
\phi^2=-I+\eta\otimes\xi, \ \ \eta\left(\xi\right)=1, \ \  g\left(\phi X,\phi Y\right) =
g\left(X,Y\right)-\eta\left(X\right)\eta\left(Y\right),
\end{gather}
for all $X,Y\in\Gamma\left(TM\right)$, where $I$ denotes the identity mapping on $TM$. From \eqref{defacm} it follows that $g(X,\phi Y)=-g(\phi X,Y)$, so that we can define the $2$-form $\Phi$ on $M$
by $\Phi\left(X,Y\right)=g\left(X,\phi Y\right)$, which is called the
\emph{fundamental $2$-form} of the almost contact metric manifold
$\left(M,\phi,\xi,\eta,g\right)$.

The manifold is said to be \emph{normal} if the tensor field
$N_{\phi}:=[\phi,\phi]_{FN}+2d\eta\otimes\xi$ vanishes identically.
%where $[-,-]_{FN}$ is the Fr\"olicher-Nijenhuis  bracket as defined in~\cite{michor}.
    Normal almost contact metric manifolds such that both $\eta$ and $\Phi$ are
closed are called \emph{coK\"{a}hler manifolds} and those such that $d\eta=c \Phi$ are called
\emph{c-Sasakian manifolds}, where $c$ is a non-zero real number (for $c=2$ one obtains the well-known \emph{Sasakian manifolds}).

    The notion of  quasi-Sasakian structure was introduced by Blair in his Ph.D. thesis in order to unify those of Sasakian and coK\"{a}hler structures. A \emph{quasi-Sasakian manifold} is defined as a
normal almost contact metric manifold whose fundamental $2$-form
is closed. A quasi-Sasakian manifold $M$ is said to be of \emph{rank $2p$} (for some $p\leq n$) if
$\left(d\eta\right)^p\neq 0$ and $\eta\wedge\left(d\eta\right)^p=0$ on $M$,
and to be of \emph{rank $2p+1$} if $\eta\wedge\left(d\eta\right)^p\neq 0$ and
$\left(d\eta\right)^{p+1}=0$  on $M$ (cf. \cite{blair0,tanno}).
Blair proved that there are no quasi-Sasakian manifolds of even rank.
Just like Blair and Tanno implicitly did, we will only consider quasi-Sasakian manifolds of constant (odd) rank.
If the rank of $M$ is $2p+1$, then the module $\Gamma(TM)$ of vector
fields over $M$ splits into two submodules as follows:
$\Gamma(TM)={\mathcal E}^{2p+1}\oplus{\mathcal E}^{2q}$, $p+q=n$, where
\begin{equation*}
{\mathcal E}^{2q}=\{X\in\Gamma(TM)\; | \; i_X d\eta=0 \mbox{ and } i_X \eta=0\}
\end{equation*}
and  ${\mathcal E}^{2p+1}={\mathcal
E}^{2p}\oplus\left\langle\xi\right\rangle$, ${\mathcal E}^{2p}$ being
the orthogonal complement of ${\mathcal
E}^{2q}\oplus\left\langle\xi\right\rangle$ in
$\Gamma\left(TM\right)$. These modules satisfy $\phi {\mathcal
E}^{2p}={\mathcal E}^{2p}$ and $\phi {\mathcal E}^{2q}={\mathcal E}^{2q}$ (\cite{tanno}).

\medskip

We now come to the main topic of our paper, i.e. $3$-quasi-Sasakian geometry,
which is framed into the more general setting of almost $3$-contact geometry.
An  \emph{almost contact metric $3$-structure} on a smooth manifold $M$  is the data of
three almost contact structures $\left(\phi_1,\xi_1,\eta_1\right)$,
$\left(\phi_2,\xi_2,\eta_2\right)$,
$\left(\phi_3,\xi_3,\eta_3\right)$ satisfying the following
relations, for any even permutation
$\left(\alpha,\beta,\gamma\right)$ of $(1,2,3)$,
\begin{equation} \label{3-sasaki}
\begin{split}
\phi_\gamma=\phi_{\alpha}\phi_{\beta}-\eta_{\beta}\otimes\xi_{\alpha}=-\phi_{\beta}\phi_{\alpha}+\eta_{\alpha}\otimes\xi_{\beta},\quad\\
\xi_{\gamma}=\phi_{\alpha}\xi_{\beta}=-\phi_{\beta}\xi_{\alpha}, \ \ \eta_{\gamma}=
\eta_{\alpha}\circ\phi_{\beta}=-\eta_{\beta}\circ\phi_{\alpha},
\end{split}
\end{equation}
and a Riemannian metric $g$ compatible with each of them. This definition was introduced, independently, by Kuo (\cite{kuo}) and Udriste (\cite{udriste}). In particular, they proved that  necessarily $\dim(M)=4n+3$. It is well
known that in any almost $3$-contact metric manifold the Reeb vector
fields $\xi_1,\xi_2,\xi_3$ are orthonormal with respect to the
compatible metric $g$ and that the structural group of the tangent
bundle is reducible to $Sp\left(n\right)\times I_3$.

Moreover, by putting ${\mathcal{H}}=\bigcap_{\alpha=1}^{3}\ker\left(\eta_\alpha\right)$ one
obtains a $4n$-dimensional \emph{horizontal distribution} on $M$ and the tangent
bundle splits as the orthogonal sum $TM={\mathcal{H}}\oplus{\mathcal{V}}$, where
${\mathcal V}=\left\langle\xi_1,\xi_2,\xi_3\right\rangle$ is the \emph{vertical distribution}.

\begin{definition}
A \emph{quasi-Sasakian $3$-structure} is an almost
contact metric $3$-structure $\{(\phi_\alpha,\xi_\alpha,\eta_\alpha,g)\}_{\alpha\in\{1,2,3\}}$ on a smooth manifold $M$
such that each almost contact metric structure is quasi-Sasakian. The manifold $M$ will be called a \emph{$3$-quasi-Sasakian manifold}.
\end{definition}
In particular, a quasi-Sasakian $3$-structure such that each structure is Sasakian is called a \emph{Sasakian $3$-structure} and the manifold is said to be a \emph{$3$-Sasakian manifold}. A quasi-Sasakian $3$-structure such that each structure is coK\"{a}hler is called a \emph{cosymplectic $3$-structure} and the manifold is said to be a \emph{$3$-cosymplectic manifold}.

%The class of $3$-quasi-Sasakian manifolds includes as special cases the well-known
%$3$-Sasakian and $3$-cosymplectic manifolds.

Let us collect some known results on $3$-quasi-Sasakian manifolds. The following theorem combines the results obtained in Theorems 3.4 and 4.2 of \cite{CapDenDil1}, and Theorem 3.7 of \cite{CapDenDil2}.

\begin{theorem}\label{classification}
Let $(M,\phi_\alpha,\xi_\alpha,\eta_\alpha,g)$ be a
$3$-quasi-Sasakian manifold. Then the $3$-dimensional distribution
$\mathcal V$ generated by $\xi_1$, $\xi_2$, $\xi_3$ is integrable.
Moreover, $\mathcal V$ defines a  Riemannian
foliation with totally geodesic leaves on $M$, and for any even permutation
$(\alpha,\beta,\gamma)$ of $(1,2,3)$ and for some $c\in \mathbb R$
\begin{equation*}
\left[\xi_\alpha,\xi_\beta\right]=c\xi_\gamma.
\end{equation*}
Moreover, $c=0$ if and only if the structure is $3$-cosymplectic.
\end{theorem}

Using Theorem \ref{classification} we may divide
$3$-quasi-Sasakian manifolds in two classes according
to the behaviour of the leaves of the
foliation $\mathcal V$: those $3$-quasi-Sasakian manifolds for which
each leaf of $\mathcal V$ is locally $SO\left(3\right)$ (or
$SU\left(2\right)$) (which corresponds to take in Theorem
\ref{classification} the constant $c\neq 0$), and those for which
each leaf of $\mathcal V$ is locally an abelian group (this corresponds
to the case $c=0$).

\section{Basic properties of $3$-quasi-Sasakian manifolds}\label{ranksection}
For a $3$-quasi-Sasakian manifold one can consider  the ranks, a priori distinct, of the three quasi-Sasakian  structures $(\phi_1,\xi_1,\eta_1,g)$, $(\phi_2,\xi_2,\eta_2,g)$, $(\phi_3,\xi_3,\eta_3,g)$.
The following theorem assures that these three ranks coincide.

\begin{theorem}[\cite{CapDenDil1},\cite{CapDenDil2}]\label{rango}
Let $(M,\phi_\alpha,\xi_\alpha,\eta_\alpha,g)$ be a
$3$-quasi-Sasakian manifold of dimension $4n+3$. Then the $1$-forms $\eta_1$, $\eta_2$
and $\eta_3$ have the same rank $4l+3$, for some integer $l\leq n$, or $1$ according to
$\left[\xi_\alpha,\xi_\beta\right]=c\xi_\gamma$ with $c\neq 0$ or
$c=0$, respectively.
\end{theorem}

According  to  Theorem \ref{rango}, the common rank of $\eta_1$, $\eta_2$, $\eta_3$ is called the \emph{rank} of the $3$-quasi-Sasakian manifold   $(M,\phi_\alpha,\xi_\alpha,\eta_\alpha,g)$

Furthermore, for any $3$-quasi-Sasakian manifold of rank $4l+3$ one can consider the following distribution
\begin{equation*}
{\mathcal E}^{4m}:=\left\{X\in\Gamma({\mathcal H})\; | \; i_X d\eta_\alpha=0, \  \alpha=1,2,3 \right\} \ \ \ (l+m=n)
\end{equation*}
and its orthogonal complement ${\mathcal E}^{4l+3}:=({\mathcal E}^{4m})^{\perp}$. In \cite{CapDenDil2} it was proved the following remarkable property of $3$-quasi-Sasakian manifolds, which in general does not hold for a
general quasi-Sasakian structure.

\begin{theorem}\label{distributions}
Let $(M,\phi_{\alpha},\xi_{\alpha},\eta_{\alpha},g)$ be a $3$-quasi-Sasakian manifold of rank $4l+3$. Then the distributions
 ${\mathcal E}^{4l+3}$ and ${\mathcal E}^{4m}$ are integrable and define Riemannian foliations with totally
geodesic leaves.
\end{theorem}

In particular it follows that $\nabla{\mathcal E}^{4l+3}\subset{\mathcal E}^{4l+3}$ and $\nabla{\mathcal E}^{4m}\subset{\mathcal E}^{4m}$. The leaves of such foliations are $3$-$c$-Sasakian manifolds (i.e., for each $\alpha\in\left\{1,2,3\right\}$, $d\eta_{\alpha}=c\Phi_{\alpha}$) and hyper-K\"{a}hler manifolds, respectively (cf. Theorem 5.4 and Theorem 5.6 of \cite{CapDenDil2}). Thus we can state the following corollary.

\begin{corollary}\label{local}
Any $3$-quasi-Sasakian manifold of rank $4l+3$, with $1\leq l < n$, is the local product of a $3$-c-Sasakian manifold  and of a hyperK\"{a}hler manifold.
\end{corollary}

Another strong consequence of Theorem \ref{distributions} is the following

\begin{corollary}
Any $3$-quasi-Sasakian manifold of maximal rank is necessarily \mbox{$3$-$c$-Sasakian}.
\end{corollary}

Thus in the two extremal cases --- maximal and minimal rank --- the geometry of a $3$-quasi-Sasakian manifold is well known. In the rank $1$ case, the structure turns out to be $3$-cosymplectic and we can refer the reader to \cite{CapDenYud1} for the main properties of these geometric structures and non-trivial examples. In the rank $(4n+3)$ case, by applying a certain homothety one can obtain a $3$-Sasakian structure.

Thus we shall deal with the non-trivial cases $\textrm{rank}(M)\neq 1$, $\textrm{rank}(M)\neq\dim(M)$.

\section{A general construction}

Let $(M',\phi'_{\alpha},\xi'_{\alpha},g')$ and
$(M'',J''_{\alpha},g'')$ be a 3-Sasakian  and a
hyper-K\"{a}hler manifold, respectively. Set $\dim(M')=4l+3$ and
$\dim(M'')=4m$. We define a canonical 3-quasi-Sasakian structure on the
product manifold $M:=M'\times M''$ in the following way.

We define as Reeb vector fields $\xi_{\alpha}:=\xi'_{\alpha}$, for
each $\alpha\in\left\{1,2,3\right\}$. Next, let $\phi_{\alpha}$ be
the $(1,1)$-tensor field determined by
\begin{equation*}
\phi_{\alpha}X:=\left\{
                 \begin{array}{ll}
                   \phi'_{\alpha}X, & \hbox{if $X\in\Gamma(TM')$} \\
                   J''_{\alpha}X, & \hbox{if $X\in\Gamma(TM'')$.}
                 \end{array}
               \right.
\end{equation*}
Finally, we consider the product metric $g:=g'+g''$ and we define
three $1$-forms $\eta_{1}$, $\eta_{2}$, $\eta_{3}$ by
$\eta_{\alpha}:=g(\cdot,\xi_\alpha)$. From the definition it follows
that the horizontal distribution ${\mathcal
H}:=\bigcap_{\alpha=1}^{3}\ker(\eta_{\alpha})$ coincides with
${\mathcal H}'\oplus{TM''}$, where ${\mathcal H}'$ is the horizontal
distribution of the $3$-Sasakian manifold $M'$. Then on $\mathcal H$
the triple $(\phi_{1},\phi_{2},\phi_{3})$ satisfies the quaternionic
relations
\begin{equation*}
\phi_{\alpha}\phi_{\beta}=-\phi_{\beta}\phi_{\alpha}=\phi_{\gamma}
\end{equation*}
for a cyclic permutation $(\alpha,\beta,\gamma)$ of
$\left\{1,2,3\right\}$. On the other hand,
$\phi_{\alpha}\xi_{\beta}=\phi'_{\alpha}\xi'_{\beta}=\xi'_{\gamma}=\xi_{\gamma}=-\phi_{\beta}\xi_{\alpha}$.
Hence
\begin{gather*}
\phi_\gamma=\phi_{\alpha}\phi_{\beta}-\eta_{\beta}\otimes\xi_{\alpha}=-\phi_{\beta}\phi_{\alpha}+\eta_{\alpha}\otimes\xi_{\beta},\\
\xi_{\gamma}=\phi_{\alpha}\xi_{\beta}=-\phi_{\beta}\xi_{\alpha}, \ \
\eta_{\gamma}=\eta_{\alpha}\circ\phi_{\beta}=-\eta_{\beta}\circ\phi_{\alpha}
\end{gather*}
and we conclude that
\(\{(\phi_\alpha,\xi_\alpha,\eta_\alpha)\}_{\alpha\in\{1,2,3\}}\)
 is an
almost contact $3$-structure on $M$. By the very definition of $g$
and $\phi_{\alpha}$ then we have that $g$ is a compatible metric.

Let us show that
\(\{(\phi_\alpha,\xi_\alpha,\eta_\alpha,g)\}_{\alpha\in\{1,2,3\}}\) is a
$3$-quasi-Sasakian structure on $M$. Notice that each fundamental
$2$-form $\Phi_\alpha:=g(\cdot,\phi_{\alpha}\cdot)$ is given by
\begin{equation*}
\Phi_{\alpha}(X,Y):=\left\{
                      \begin{array}{ll}
                        \Phi'_{\alpha}(X,Y), & \hbox{if $X,Y\in\Gamma(TM')$} \\
                        0, & \hbox{if $X\in\Gamma(TM')$, if $Y\in\Gamma(TM'')$} \\
                        \Omega''_{\alpha}(X,Y), & \hbox{if $X,Y\in\Gamma(TM'')$}
                      \end{array}
                    \right.
\end{equation*}
where $\Phi'_\alpha$ and $\Omega''_{\alpha}$ denote the fundamental
$2$-forms of $(M',\phi'_\alpha,\xi'_\alpha,g')$ and
$(M'',J''_\alpha,g'')$, respectively. By using the well-known formula
\begin{align*}
d\Phi_{\alpha}(X,Y,Z)&=X(\Phi_{\alpha}(Y,Z))+Y(\Phi_{\alpha}(Z,X))+Z(\Phi_{\alpha}(X,Y))\\
&\quad-\Phi_\alpha([X,Y],Z)-\Phi_\alpha([Y,Z],X)-\Phi_\alpha([Z,X],Y)
\end{align*}
we see that
\begin{equation*}
d\Phi_{\alpha}(X,Y,Z)=\left\{
                       \begin{array}{ll}
                         d\Phi'_{\alpha}(X,Y,Z), & \hbox{if $X,Y,Z\in\Gamma(TM')$} \\
                         0, & \hbox{if $X,Y\in\Gamma(TM')$, $Z\in\Gamma(TM'')$} \\
                         0, & \hbox{if $X\in\Gamma(TM')$, $Y,Z\in\Gamma(TM'')$} \\
                         d\Omega''_{\alpha}(X,Y,Z), & \hbox{if $X,Y,Z\in\Gamma(TM'')$.}
                       \end{array}
                    \right.
\end{equation*}
Since $\Phi'_\alpha$ and $\Omega''_{\alpha}$ are closed, we conclude
that also each $\Phi_\alpha$ is closed. \ Moreover, in order to
prove the normality of the $3$-structure
\(\{(\phi_{\alpha},\xi_{\alpha},\eta_{\alpha})\}_{\alpha\in\{1,2,3\}}\), it is enough to check the vanishing of
$N_{\phi_\alpha}$
on the couples of vector fields of this type:
\begin{equation*}
N_{\phi_\alpha}(X',Y'), \ \ \ \ N_{\phi_\alpha}(X',Y''), \ \ \
N_{\phi_\alpha}(Y',Y''),
\end{equation*}
where $X'$, $Y'$  are vector fields on $M'$ and $X''$, $Y''$  are
vector fields on $M''$. Since $d\eta_\alpha=0$ on $TM''$, using the
definitions of $\phi_\alpha$ and $N_{\phi_\alpha}$,
\begin{gather*}
N_{\phi_\alpha}(X',Y')=N_{\phi'_\alpha}(X',Y')=0,\\
N_{\phi_\alpha}(X'',Y'')=N_{J''_\alpha}(X'',Y'')=0,
\end{gather*}
because $M'$ is 3-Sasakian and $M''$ hyper-K\"{a}hler, and
\begin{equation*}
N_{\phi_\alpha}(X',Y'')=\phi_\alpha^{2}[X',Y'']+[\phi_\alpha
X',\phi_\alpha Y'']-\phi_\alpha[\phi_\alpha
X',Y'']-\phi_\alpha[X',\phi_\alpha Y'']=0
\end{equation*}
since each summand in the last equation is zero.

Therefore $(M,\phi_\alpha,\xi_\alpha,\eta_\alpha,g)$ is a
3-quasi-Sasakian manifold with rank $4l+3=\dim(M')$.

We say that $f\colon M\to M$ is a 3-quasi-Sasakian isometry if
it is an isometry of the Riemannian manifold $(M,g)$
preserving each quasi-Sasakian structure, namely
\begin{equation}\label{isometry1}
f_\ast \circ \phi_\alpha = \phi_\alpha \circ f_\ast, \ \ \ f_\ast
\xi_\alpha = \xi_\alpha
\end{equation}
for each $\alpha\in\left\{1,2,3\right\}$. Notice that from
\eqref{isometry1} it follows that
\begin{equation}\label{isometry2}
f^{\ast}\eta_\alpha = \eta_\alpha.
\end{equation}
Indeed for any $X\in\Gamma(TM)$
\begin{equation*}
f^{\ast}\eta_\alpha(X)=\eta_{\alpha}(f_{\ast}X)=g(f_{\ast}X,\xi_\alpha)=g(f_{\ast}X,f_{\ast}\xi_\alpha)=g(X,\xi_\alpha)=\eta_\alpha(X).
\end{equation*}
Given a free and properly discontinuous action of a discrete group
(in particular, a free  action of a finite group) $G$ on a 3-quasi-Sasakian manifold $M$ by 3-quasi-Sasakian isometries,
 the quotient $M/G$ is a smooth manifold of the same dimension as $M$ and inherits a 3-quasi-Sasakian structure from $M$.

Recall that $f\colon M''\to M''$ is a hyper-K\"ahler isometry if
$f$ is an isometry of the Riemannian manifold $(M'',g'')$ and
\begin{equation}\label{isometry3}
f_\ast \circ J''_\alpha = J''_\alpha \circ f_\ast
 \end{equation}
for each $\alpha\in\left\{1,2,3\right\}$.
From \eqref{isometry3} it follows that
\begin{equation*}
f^* \Omega''_\alpha = \Omega''_\alpha.
\end{equation*}

Suppose $G$ is a finite group that acts on $M'$ by 3-Sasakian isometries
and on $M''$ by hyper-K\"ahler isometries. Then $G$ also acts on the product
manifold $M'\times M''$ by $g\cdot (p',p'') = (g\cdot p',g\cdot p'')$,
$g\in G$. It is easy to check that $G$ preserves the 3-quasi-Sasakian
structure on $M'\times M''$ defined above. If the action of $G$ on $M'\times
M''$ is free then the quotient $(M'\times M'')/G$ is a
a 3-quasi-Sasakian manifold.
\bigskip

As an application, we consider the 3-Sasakian manifold
$S^{4l+3}$. We recall how the standard 3-Sasakian structure
$(\phi'_\alpha,\xi'_\alpha,\eta'_\alpha,g')$ of the sphere is
defined . Let us consider the sphere $S^{4l+3}$ as an
hypersurface in $\mathbb{H}^{l+1}$. Let $(J_{1},J_{2},J_{3})$ be the
standard quaternionic structure of $\mathbb{H}^{l+1}$ that is
upon identification of $T_x \mathbb{H}^{l+1}$
with $\mathbb{H}^{l+1}$ the operators $J_1$, $J_2$, $J_3$ act by multiplication
with $\qi$, $\qj$, $\qk$ on the left.

Let $N$ be the outer
vector field normal to the sphere. Then one can prove that the
vector fields
\begin{equation}\label{sphere1}
\xi'_\alpha:=-J_{\alpha}N
\end{equation}
are tangent to the sphere. Moreover, for any
$X\in\Gamma(TS^{4l+3})$, we decompose $J_{\alpha}X$ in
their components tangent and normal to the sphere,
\begin{equation}\label{sphere2}
J_{\alpha}X = \phi'_{\alpha}X + \eta'_{\alpha}(X)N,
\end{equation}
so obtaining, for each $\alpha\in\left\{1,2,3\right\}$, a tensor
field $\phi'_\alpha$ and a $1$-form $\eta'_{\alpha}$ on
$S^{4l+3}$. Then one can check that the geometric structure
\(\{(\phi'_\alpha,\xi'_\alpha,\eta'_\alpha,g')\}_{\alpha\in\{1,2,3\}}\) is a 3-Sasakian structure
on $S^{4l+3}$, being $g'$ the
Riemannian metric induced by the Riemannian metric $g$ of
$\mathbb{H}^{l+1}\cong \mathbb{R}^{4l+4}$.

Now we consider the isometry $f$ of $\mathbb{H}^{l+1}$ given by the
multiplication with $\textbf{i}$ {on the right}. Notice that $f(S^{4l+3})=S^{4l+3}$,
because for any $x\in S^{4l+3}$ one has $\|f(x)\|=\|x\qi \|=\|x\|=1$.
Hence $f$ induces an isometry on $(S^{4l+3},g')$, again denoted by
$f$. Notice that the associativity of the product in $\mathbb{H}$ implies
\begin{equation*}
f_{\ast}\circ J_{\alpha}=J_{\alpha}\circ f_{\ast}.
\end{equation*}
Thus $f$ is a hyper-K\"ahler isometry.
Moreover, for any $X\in\Gamma(TS^{4l+3})$,
$g(f_{\ast}N,f_{\ast}X)=g(N,X)=0$, so that $f_{\ast}N\in
(TS^{4l+3})^{\perp}=\left\langle N\right\rangle$. Since
$\left\| N\right\|=1$ and $f$ is an isometry, it follows that
\begin{equation*}
f_{\ast}N=N.
\end{equation*}
Then by \eqref{sphere1} and \eqref{sphere2} we get
\begin{equation*}
f_{\ast}\xi'_{\alpha}=-f_{\ast}J_{\alpha}N=-J_{\alpha}f_{\ast}N=-J_{\alpha}N=\xi'_{\alpha},
\end{equation*}
and, for all $X\in\Gamma(TS^{4l+3})$,
\begin{align*}
f_{\ast}(\phi'_\alpha X)+\eta'_\alpha(X)N&=f_{\ast}(\phi'_\alpha
X)+\eta'_\alpha(X)f_{\ast}N\\
&=f_{\ast}J_{\alpha}X\\
&=J_{\alpha}f_{\ast}X\\
&=\phi'_\alpha(f_\ast X)+\eta'_\alpha(f_\ast X)N,
\end{align*}
from which, taking the tangential and the normal components to the
sphere, it follows that
$f_{\ast}\circ\phi'_{\alpha}=\phi'_{\alpha}\circ f_{\ast}$ and
$f^{\ast}\eta'_{\alpha}=\eta'_{\alpha}$. Thus $f$ is a 3-Sasakian
isometry of $S^{4l+3}$.
Moreover, $f^4$ is the identity operator. Thus we get an action of $\Z_4$ on
$S^{4l+3}$ by 3-Sasakian isometries.

Let $m$ be a positive integer. We denote the hyper-K\"ahler isometry of
$\mathbb{H}^m$, $(q_1,\dots,q_m) \mapsto (q_1\qi,\dots, q_m\qi) $, by $h$.
The map $h$ induces a hyper-K\"ahler isometry on the torus $\tor^{4m} =
\mathbb{H}^m/ \Z^{4m}$.
Thus $h$ generates an action of $\Z_4$ on $\tor^{4m}$ by hyper-K\"ahler isometries.
Note, that $\Z_4$ acts freely on $S^{4l+3}$, but has a
fixed point $[0]$ in $\T^{4m}$. Nevertheless, the resulting action of $\Z_4$ on
$S^{4l+3}\times \tor^{4m}$ is free.
We will denote the 3-quasi-Sasakian manifold $(S^{4l+3} \times \tor^{4m})/\Z_4$ by
$M_{l,m}$.
%and
%\begin{equation}\label{mainexample}
%M^{4n+3}_{f}=(S^{4l+3}\times\mathbb{R}^{4m})/\mathbb{Z}^{4m}
%\end{equation}
%is a compact 3-quasi-Sasakian manifold of rank $4l+3$.

Concerning this example, in view of Corollary \ref{local}, an interesting
question is the following: is $M_{l,m}$ the global product of a $3$-Sasakian
manifold of dimension $4l+3$ and a hyperK\"{a}hler manifold of dimension
$4m$?

In the next Section we shall show that the answer is negative, at least in the case $l=1$ and $m=1$. Namely we will prove that the 3-quasi-Sasakian manifold
\begin{equation*}
M_{1,1}:=(S^{7}\times\mathbb{T}^{4})/\mathbb{Z}_{4}
\end{equation*}
is not
topologically equivalent to the product of a $7$-dimensional compact
3-Sasakian manifold and a $4$-dimensional compact hyper-K\"{a}hler
manifold.

\section{The manifold $M_{1,1}=(S^{7}\times\mathbb{T}^{4})/\mathbb{Z}_{4}$}
Let $M$ be a compact Riemannian manifold and $G$ a finite group freely
acting on $M$. Denote by $\rho_M$ the corresponding group homomorphism from  $G$ to $\aut (M)$. Then from the Hodge theory we  obtain the isomorphism
\begin{equation}
\label{iso}
H^*\left( M/G \right)\cong H^*\left( M \right)^G:=\left\{ x \in H^*( M )\,  \middle|\, \rho(a)^* x=x, \ \mbox{for all }a\in G\right\}.
\end{equation}
Indeed, every harmonic form on $M/G$ lifts a $G$-periodic harmonic
form on $M$ and every $G$-periodic form on $M$ defines a periodic
form on $M/G$. Here it is important that the projection $M\to M/G$ is
a local diffeomorphism and the Laplacian $\triangle$ is defined
locally.

Now, let $M$ and $N$ be two compact manifolds with $G$-action given by
$\rho_M \colon G \to \aut (M)$ and $\rho_N \colon G \to \aut(N)$.
We will write $\rho\colon G\to \aut(M\times N)$ for the corresponding action on
the product $M\times N$.
% Then
%$G$ acts on the product $M\times N$ by
%$$
%g\left( x,y \right) := \left( gx,gy \right).
%$$
If $\omega$ is a $q$-form on $M$ and $\sigma$ is a $p$-form on $N$, then
$\pp_M^* \omega\wedge \pp_N^* \sigma$  is a $(p+q)$-form on $M\times N$. Moreover,
\begin{equation*}
\rho(a)^*  \left(\pp_M^* \omega\wedge\pp_N^* \sigma  \right) = \pp_M^*
\rho_M(a)^*   \omega \wedge
\pp_M^* \rho_N(a)^*  \sigma
\end{equation*}
 for $a\in G$.
By K\"unneth theorem we have
\begin{equation*}
H^k\left( M\times N \right) = \bigoplus_{p+q= k} H^q\left( M
\right)\otimes H^p\left( N \right).
\end{equation*}
From the above we see that $H^q\left( M \right)\otimes H^p\left( N
\right)$ is a $G$-invariant subspace of $H^k\left( M \times N \right)$.
Therefore
\begin{equation}
\label{kun}
H^k\left( M\times N \right)^G = \bigoplus_{q+p= k} \left( H^q\left(
M \right)\otimes H^p\left( N \right) \right)^G.
\end{equation}
Let us now specialize to the case of $M = S^{7}$ and $N = \tor^{4}$ with the
action of $\Z_4$ on $S^7$ and $\tor^4$ defined in the previous section.
Note that since the isometry $f\colon S^7\to S^7$ was orientation preserving,
the induced action of $\Z_4$ on $H^7(S^7) \cong \R$ is trivial.
It is also clear that $\Z_4$ acts trivially on $H^0(S^7) \cong \R$. Thus
for any $0\le k\le 4$
\begin{align}
\label{hs}
(H^0(S^7) \otimes H^k(\tor^4))^{\Z_4} & \cong H^k(\tor^4)^{\Z_4}, &
(H^7(S^7) \otimes H^k(\tor^4))^{\Z_4} & \cong H^k(\tor^4)^{\Z_4}.
\end{align}
Let us denote the Betti numbers of $M_{1,1}$ by $b_k$  and we write
$\tilde{b}_k$  for $\dim
H^k(\tor_4)^{\Z_4}$.
Then, from \eqref{iso}, \eqref{kun}, and \eqref{hs} it follows that
\begin{equation}
\label{bettis}
\begin{aligned}
b_0 & = \tilde{b}_0 = 1, & b_1 & = \tilde{b}_1, & b_2 & = \tilde{b}_2, & b_3 & =
\tilde{b}_3, & b_4 & = \tilde{b}_4, & b_5 & = b_6=0,\\  b_7 &= \tilde{b}_0 = 1,
&
b_8 & = \tilde{b}_1, & b_9 & = \tilde{b}_2, & b_{10} & = \tilde{b}_3 ,&
b_{11} & = \tilde{b}_4 = 1.
\end{aligned}
\end{equation}
Now we compute $\tilde{b}_1$, $\tilde{b}_2$,  and $\tilde{b}_3$. Note that from
the above equations and Poincar\'e duality for $M_{1,1}$, we get
$\tilde{b}_1 = b_1 = b_{10} = \tilde{b}_3$.
Thus it is enough to compute $\tilde{b}_1$ and~$\tilde{b}_2$.
The cup product on $H^*(\tor^4)$ induces the $\Z_4$-invariant isomorphism
\begin{equation*}
\begin{aligned}
\Lambda^* H^1(\tor^4) &\longrightarrow H^*(\tor^4)\\
[\alpha_1] \wedge \dots \wedge [\alpha_k] & \longmapsto
[\alpha_1 \wedge \dots \wedge \alpha_k],
\end{aligned}
\end{equation*}
where $\Lambda^* V$ stands for the exterior algebra of a vector space $V$.
Thus
\begin{equation*}
\tilde{b}_k = \dim (\Lambda^k H^1(\tor^4))^{\Z_4} = \dim (\Lambda^k
H^1(\tor_4))^{h^*},
\end{equation*}
where $h\colon \tor^4\to \tor^4$ was defined in the previous section.
Let $x_{\qq}$, $x_{\qi}$, $x_{\qj}$, $x_{\qk}$ be the coordinate functions on
$\tor^4$ induced from $\mathbb{H}$. Denote by
Let $\theta_{\qq}$, $\theta_{\qi}$, $\theta_{\qj}$, and $\theta_{\qk}$ be the
dual 1-forms. Then the classes
$[\theta_{\qq}]$, $[\theta_{\qi}]$, $[\theta_{\qj}]$, and $[\theta_{\qk}]$ give a
basis of  $H^1(\tor^4)$.
The matrix of $h^*$ in this basis is
\begin{equation*}
A := \left(
\begin{array}{cccc}
0 & 1 & 0 & 0 \\
-1 & 0 & 0 & 0\\
0 & 0 & 0 & -1 \\
0 & 0 & 1 & 0
\end{array}
\right)
\end{equation*}
The eigenvalues of  $A$ over $\mathbb{C}$ are $i$ and $-i$. Since $1$ is not
among the eigenvalues there is no
element in $H^1(\tor^4)$ which is $h^*$-invariant.  Thus $\tilde{b}_1 = 0$.
The matrix of $h^*$ in the basis
\begin{align*}
[\theta_{\qq}] \wedge [\theta_{\qi}], \quad
[\theta_{\qq}] \wedge [\theta_{\qj}] + [\theta_{\qi}] \wedge [\theta_{\qk}], \quad
[\theta_{\qq}] \wedge [\theta_{\qj}] - [\theta_{\qi}] \wedge [\theta_{\qk}]
\\[2ex]
[\theta_{\qq}] \wedge [\theta_{\qk}] + [\theta_{\qi}] \wedge [\theta_{\qj}], \quad
[\theta_{\qq}] \wedge [\theta_{\qk}] - [\theta_{\qi}] \wedge [\theta_{\qj}], \quad
[\theta_{\qj}] \wedge [\theta_{\qk}]
\end{align*}
of $\Lambda^2 H^1(\tor^4)$ is
\(\textrm{diag}\left( 1,-1,1,1,-1,1 \right)\).
Thus $\tilde{b}_2 = \dim (\Lambda^2 H^1(\tor^4))^{h^*} =  4$.
Using that $\tilde{b}_1 = 0 = \tilde{b}_3$, $\tilde{b}_2 = 4$, we get
from \eqref{bettis}
\begin{equation*}
\begin{aligned}
b_0 &= b_4 = b_7 = b_{11} = 1, & b_1 &= b_3 = b_5 = b_6 = b_8 =  b_{10 } = 0, &
b_2 & = b_9 = 4.
\end{aligned}
\end{equation*}
Thus the Poincar\'e polynomial of $M_{1,1}$ is
\begin{equation}
P(t) := 1 + 4t^2 + t^4 + t^7 + 4t^9 + t^{11} = (1 + t^7) (1 + 4t^2 + t^4).
\end{equation}

Suppose $M_{1,1}\cong M'\times M''$, where $M'$ is a $7$-dimensional 3-Sasakian
manifold and $M''$ a $4$-dimensional hyper-K\"ahler manifold.
Denote by $P'$ and $P''$ the Poincar\'e polynomial of $M'$, respectively of
$M''$. Then by K\"unneth theorem
\begin{equation}\label{pord}
 P(t) = P'(t) P''\left(t  \right).
\end{equation}
We will write $p_1 \le p_2$ for two polynomials with non-negative coefficients
if all the coefficients of $p_2 - p_1$
are non-negative. We also write $p_1 < p_2$ if $p_1 \le p_2 $ and $p_1\not=
p_2$.
It is obvious that if $p_1 \le p_2$ then $p_1 p \le p_2 p$, and if $p_1<p_2$
then $p_1 p <p_2 p $ for any non-zero polynomial $p$ with non-negative
coefficients.

With this notation we have $P''(t) \ge 1 + t^7$, since $M''$ is a compact
orientable
$7$-dimensional manifold.
Let us recall the following well-known
result.
\begin{theorem}
    \label{kodaira}
    If $M^4$ is a compact four-dimensional hyper-K\"{a}hler manifold, then $M^4$ is either a K3 surface or a
    four dimensional torus.
\end{theorem}
\begin{proof}
    From \cite[Theorem 8.1]{wakakuwa} it follows that
    $b_1(M^4)$ is even. Moreover, since every hyper-K\"{a}hler manifold is
    Calabi-Yau, $M^4$ has a trivial
    canonical bundle.  Therefore, by the Kodaira classification
    (cf.~\cite[Section 6A]{kodaira1}) $M^4$ is either a K3 surface or a 4-torus.
\end{proof}
If $M''\cong \tor^4$, then $P''(t) = 1 + 4t + 6t^2 + 4t^3 +t^4 $. If $M''$ is a
K3 surface then
$P''(t) = 1 +22t^2 +t^4$. Thus in both cases $P''(t) > 1+4t^2 + t^4$.
Therefore
\begin{equation*}
P(t) = P'(t) P''(t) > (1+t^7) (1+4t^2 + t^7) =
P(t),
\end{equation*}
which gives a contradiction to our assumption $M_{1,1}\cong M'\times M''$.
So we finally proved
\begin{theorem}\label{main}
There exist an $11$-dimensional compact 3-quasi-Sasakian manifold of rank $7$ which is
not a
global product of a $7$-dimensional 3-Sasakian manifold and a $4$-dimensional
hyper-K\"ahler manifold.
\end{theorem}

%%%%%%%%%%%%%%%%%%%%%%%%%%%%%%%%%

\end{document}